\documentclass[11pt,a4paper,english]{amsart}
\usepackage[utf8]{inputenc}
\usepackage[english]{babel}

\usepackage{amsmath} 
\usepackage{amsthm} 
\usepackage{graphicx} 
\usepackage{xcolor} 
\usepackage{amsfonts}
\usepackage[biblabel]{cite}
\usepackage{bm} 
\usepackage[hidelinks]{hyperref} 
\usepackage{amssymb} 
\usepackage{tikz-cd} 
\usepackage{amscd}
\usepackage{etoolbox}
\patchcmd{\thmhead}{(#3)}{#3}{}{}
\usepackage{enumitem}
\usepackage{breqn} 
\usepackage{faktor} 
\usepackage{xfrac} 
\usepackage{tikz}
\usetikzlibrary{matrix,shapes,arrows,positioning,chains, calc,trees,backgrounds}
\usepackage{comment} 
\usepackage{pdfpages}
\usepackage{xr}
\usepackage{subcaption}

\DeclareMathOperator{\ini}{in} 
\DeclareMathOperator{\reg}{reg} 
\DeclareMathOperator{\N}{N} 

\DeclareMathOperator{\Span}{Span}

\DeclareMathOperator{\den}{den} 
\DeclareMathOperator{\lcm}{lcm} 

\DeclareMathOperator{\FB}{FB}

\newcommand{\F}{{\mathbb{F}}}
\newcommand{\fq}{\mathbb{F}_q}

\newcommand{\PP}{{\mathbb{P}}}

\newcommand{\NN}{{\mathbb{N}}}
\newcommand{\Z}{{\mathbb{Z}}}
\newcommand{\A}{{\mathbb{A}}}

\newcommand{\calI}{{\mathcal{I}}}

\newcommand{\set}[1]{\left\{#1\right\}}
\newcommand{\size}[1]{\left|#1\right|}
\newcommand{\gen}[1]{\left\langle #1 \right\rangle}
\newcommand{\Floorfrac}[2]{\left\lfloor \frac{#1}{#2} \right\rfloor}

\newcommand{\bfa}{\mathbf{a}}
\renewcommand{\bar}{\overline}
\newcommand{\Mon}{\bar{\mathbb{M}}}
\newcommand{\Moni}[2]{\bar{\mathbb{M}}_{#1}^{(#2)}}
\newcommand{\M}{\mathbb{M}}
\newcommand{\FBi}[1]{\FB^{(#1)}}

\usepackage{mathtools} 
\DeclarePairedDelimiter\abs{\lvert}{\rvert}%
\DeclarePairedDelimiter\norm{\lVert}{\rVert}%

\usepackage{array}
\usepackage{tabulary}
\newcolumntype{C}[1]{>{\centering\arraybackslash}p{#1}}

\makeatletter
\let\oldabs\abs
\def\abs{\@ifstar{\oldabs}{\oldabs*}}
\let\oldnorm\norm
\def\norm{\@ifstar{\oldnorm}{\oldnorm*}}
\makeatother

\newtheorem{thm}{Theorem}[section]
\newtheorem{prop}[thm]{Proposition}
\newtheorem{cor}[thm]{Corollary}
\newtheorem{lem}[thm]{Lemma}
\theoremstyle{definition}
\newtheorem{defn}[thm]{Definition} 

\newtheorem{rem}[thm]{Remark} 
 
\newtheorem{ex}[thm]{Example}
 
\newtheorem{conj}[thm]{Conjecture} 

\title[Maximum number of zeroes of polynomials on weighted projective spaces]{Maximum number of zeroes of polynomials on weighted projective spaces over a finite field}

\author[J. Nardi]{Jade Nardi}
\address[Jade Nardi]{Univ Rennes, CNRS, IRMAR - UMR 6625, Rennes Cedex, France.}
\email{jade.nardi@univ-rennes.fr}

\author[R. San-José]{Rodrigo San-José}
\address[Rodrigo San-José]{Department of Mathematics\\ Virginia Tech\\ Blacksburg, VA USA. \newline \textit{Previous address:} IMUVA-Mathematics Research Institute, Universidad de Valladolid, 47011 Valladolid (Spain).}
\email{rsanjose@vt.edu}



\begin{document}

\maketitle

\begin{abstract}
We compute the maximum number of rational points at which a homogeneous polynomial can vanish on a weighted projective space over a finite field, provided that the first weight is equal to one. This solves a conjecture by Aubry, Castryck, Ghorpade, Lachaud, O'Sullivan and Ram, which stated that a Serre-like bound holds with equality for weighted projective spaces when the first weight is one, and when considering polynomials whose degree is divisible by the least common multiple of the weights. We refine this conjecture by lifting the restriction on the degree and we prove it using footprint techniques, Delorme's reduction and Serre's classical bound. 
\end{abstract}

\section{Introduction}

Let $\F$ denote a field and let $\overline{\F}$ be its algebraic closure. Let $w=(w_0,w_1,\dots,w_m) \in \NN_{\geq 1}^{m+1}$. The \emph{weighted projective space} (WPS for short) of weight $w$, denoted by $\PP(w)$, over the field $\F$ is defined as the quotient
\[ \PP(w)=(\A^{m+1}\setminus \{(0,\dots,0)\})/\overline{\F}^*\] under the following action of $\overline{\F}^*$:
$\lambda \cdot (x_0,\dots,x_m)= (\lambda^{w_0} x_0,\dots,\lambda^{w_m}x_m)$ for $\lambda \in \overline{\F}^*$.
This generalizes the notion of projective spaces, which are recovered with $w=(1,\dots,1)$. WPS are instances of projective toric varieties, with finite quotient singularities when $w\neq(1,\dots,1)$. As recalled in \cite{Beltrametti}, WPS arise naturally in diverse contexts in algebraic geometry, such as the study of invariants and moduli spaces, notably of surfaces (see, e.g., \cite{Catanese,Dolgachev_surfaces}). 

In the whole paper, we assume that $\F=\fq$ is the finite field with $q$ elements. The set of $\fq$-rational points of $\PP(w)$, that are the elements of $\PP(w)$ which are invariant under the Frobenius automorphism, is denoted by $\PP(w)(\fq)$. The number of $\fq$-points of $\PP(w)$ does not depend on the vector of weights $w$ and we set 
$$
p_m:=\abs{\PP(w)(\fq)}=\frac{q^{m+1}-1}{q-1}.
$$
The coordinate ring (or Cox ring) of $\PP(w)$ is the polynomial ring $\fq[x_0,\dots,x_m]$ where $\deg(x_i)=w_i$. We denote this ring by $\fq[x_0,\dots,x_m]^w$ to highlight the role of $w$. For any integer $d \geq 0$, we write $\fq[x_0,\dots,x_m]_d^w$ for the degree-$d$ component of $\fq[x_0,\dots,x_m]^w$, the vector space that consists of the polynomials $f \in \fq[x_0,\dots,x_m]^w$ that have degree $d$ and are \emph{homogeneous} with respect to the grading associated to $w$ (or \emph{weighted homogeneous}), i.e.,
\[f(\lambda^{w_0}x_0,\dots,\lambda^{w_m}x_m)=\lambda^{d}f(x_0,\dots,x_m).\]

For any homogeneous polynomial $f \in \fq[x_0,\dots,x_m]^w$, we denote by $V_{\PP(w)}(f)$ the hypersurface defined by $f=0$, by $V_{\PP(w)}(f)(\fq)$ the set of $\fq$-points of $\PP(w)$ at which $f$ vanishes, and by $\abs{V_{\PP(w)}(f)(\fq)}$ its cardinality.

Aubry et al. \cite{ghorpadeWPRM} set
\begin{equation}
    e_q(d;w_0,w_1,\dots,w_m):=\max_{f \in \fq[x_0,\dots,x_m]_d^w} \abs{V_{\PP(w)}(f)(\fq)},
\end{equation}
and stated the following conjecture \cite[Conjecture~2.3]{ghorpadeWPRM}.

\begin{conj}\label{sudhir conjecture}
Let $w'=(w_1,\dots,w_m)$. If $w_0=1$ and $\lcm(w')\mid d$, then 
$$
e_q(d;1,w_1,\dots,w_m)= \min\left \{p_m,\frac{d}{w_1}q^{m-1}+p_{m-2}\right\}.
$$
\end{conj}

When $w_0=w_1=\cdots=w_m=1$, this result was already conjectured by Tsfasman at the \textit{Journées Arithmétiques de Luminy} in 1989, and proven in 1991 by Serre \cite{serreZeroesProjectiveSpace} and S{\o}rensen \cite{sorensen}, independently. Conjecture \ref{sudhir conjecture} has been proven for $m\leq 2$ by Aubry et al. \cite[Theorem~2.4]{ghorpadeWPRM}. Aubry and Perret \cite{aubryperretWPRM} provide a proof for Conjecture \ref{sudhir conjecture} in the case where $w_0=w_1=1$, which we pointed out to be valid only if $\gcd(w_i,w_j,q-1)= 1$ for $i \neq j$ \cite{aubryperretWPRM:corrigendum}.

Conjecture~\ref{sudhir conjecture} is strongly connected with algebraic coding theory, as it provides the minimum distance of weighted projective Reed-Muller codes \cite{ghorpadeWPRM} (note that this definition differs from that in \cite{sorensenWRM}), or simply  projective Reed-Muller codes when $w=(1,\dots,1)$ \cite{sorensen}. The minimum distance is the most important basic parameter, since it determines the error-correction capability of the code. It is also a key parameter for many applications of coding theory (e.g., code-based cryptography \cite{McEliece1978}, quantum error-correction \cite{kkks}, private information retrieval \cite{hollanti_pir}, secret sharing \cite{cramer_secret_sharing_and_codes}, multiparty computation \cite{cramer_multiparty_computation}).

\medskip

In what follows we assume that
\[1=w_0 \leq w_1=\dots=w_\ell < w_{\ell+1} \leq \cdots \leq w_m\]
with $1\leq \ell \leq m-1$. As a consequence, we have $w_m\geq 2$. The main goal of this paper is to prove the following theorem.

\begin{thm}\label{t:theorem_conjecture_intro}
Assume that $w_0=1$. Then 
\[
e_q(d;1,w_1,\dots,w_m)= \min\set{p_m,\left(\Floorfrac{d-1}{w_1}+1\right)q^{m-1}+p_{m-2}}.
\]
\end{thm}

If $w_1$ divides $d$, then 
$
\Floorfrac{d-1}{w_1}=\frac{d}{w_1}-1.
$
Thus, Theorem \ref{t:theorem_conjecture_intro} refines Conjecture \ref{sudhir conjecture} by removing the divisibility conditions on the degree $d$. Note that assuming that $\lcm(w')$ divides $d$ avoids many degenerate cases. One can easily check that if $w_i \nmid d$ for some $0\leq i \leq m$, then no monomial of the form $x_i^\alpha$ can have weighted degree $d$. This implies that the point $(0:\dots:0:1:0:\dots:0)$, with a single 1 in position $i$, is a common zero of all the homogeneous polynomials in $\fq[x_0,\dots,x_m]^w_d$. However, our argument does not require this condition and we are thus able to obtain a more general result.

An important tool for our proof is the projective footprint bound for weighted projective spaces, which is similar to the approach taken by Beelen et al. \cite{projectivefootprint} to prove Conjecture \ref{sudhir conjecture} for the case $w=(1,\dots,1)$ and by \c{C}ak{\i}ro{\u{g}}lu et al., which proved Theorem~\ref{t:theorem_conjecture_intro} for $m=2$ \cite[Theorem~6.1]{caki_sahin_nardi25}. This technique comes from Gr\"obner basis theory and provides an upper bound for the number of zeroes of a polynomial in terms of a combinatorial quantity associated to its initial term, called the \emph{footprint}. For the affine space, this is a well-known approach (see, e.g., \cite[Thm. 6 and Prop. 7, Chapter 5 §3]{cox}) and it has been used in an analogous way to Bezout's theorem \cite{geilbezout}. These ideas have been generalized to the projective space \cite{projectivefootprint,martinez} and more generally to simplicial toric varieties \cite{nardi_projective_toric}. Nevertheless, the footprint bound alone cannot be used to prove the main theorem, as the footprint of some monomials overestimates the number of zeroes compared to the sought bound. Beelen et al. \cite{projectivefootprint} benefited from the $(m+1)$-transitivity of the automorphism group of the projective space $\PP^m$ to overcome this issue. In our case, the footprint technique is fruitful only for polynomials whose initial monomial involves $x_0$ or at least one variable of weight greater than $w_1$ (see Lemmas \ref{lem:FB_x0} and \ref{lem:ws>1}). Otherwise, such a polynomial $f$ only involves the variables $x_0,\dots,x_\ell$, and we use Delorme's reduction (Lemma \ref{l:delorme}) and Serre's bound on $\PP^\ell$ to estimate $\abs{V_{\PP(w)}(f)(\fq)}$ (see Lemma \ref{lem:var_weights=w_1}).

The paper is organized as follows. In Section \ref{s:preliminaries} we provide the necessary preliminaries for our proof, which mainly include facts about the vanishing ideal of $\PP(w)(\fq)$ and the footprint bound. In Section \ref{s:proof} we prove Theorem \ref{t:theorem_conjecture_intro}, and in Section \ref{s:w0>1} we study how $e_q(d;w_0,\dots,w_m)$ behaves when considering $w_0>1$. We compute $e_q(d;w_0,w_1)$ for any weights and we obtain lower and upper bounds for $e_q(d;w_0,\dots,w_m)$. Moreover, we provide many examples and numerical experiments showing how different the behaviour of $e_q(d;w_0,\dots,w_m)$ is when $w_0>1$. Finally, in Appendix \ref{s:appendix} we provide a precise description of two classical isomorphisms of WPS (weight reductions), and how $\fq$-points and polynomials are mapped under these, which we use throughout the paper.

\section{Preliminaries}\label{s:preliminaries}

We denote by $\M$ the set of all the monomials of $\fq[x_0,\dots,x_m]^w$. We also set
\[\M_d:=\set{x_0^{a_0}\cdots x_m^{a_m} \in\M : \deg(x_0^{a_0}\cdots x_m^{a_m})=d}, \]
and we can write $\fq[x_0,\dots,x_m]^w_d=\Span \M_d$.

\begin{defn}
    Let $w=(w_0,w_1,\dots,w_m) \in \NN_{\geq 1}^{m+1}$. We denote by $\gen{w_0,w_1,\dots,w_m}_\NN$ (or $\gen{w}_\NN$ for short) the semigroup of integers that can be written as a linear combination of $w_0,w_1,\dots,w_m$ with nonnegative integer coefficients.  
\end{defn}

From this definition, it is clear that $\M_d \neq \emptyset$ if and only if $d\in\gen{w}_\NN$. The cardinality of $\M_d$ is equal to the denumerant of $d$ with respect to $w$, defined as follows.

\begin{defn}\label{def:denum}
    Let $w=(w_0,w_1,\dots,w_m) \in \NN_{\geq 1}^{m+1}$. For any $d \in \NN$, we define the \emph{denumerant of $d$ with respect to $w$} as
    \[\den(d;w):=\size{\set{(i_0,\dots,i_{m}) \in \NN^{m+1} \text{ such that } w_0 i_0 + \dots+ w_mi_m=d}}.\]
\end{defn}

The \emph{vanishing ideal of $\PP(w)(\fq)$}, denoted by $\calI(\PP(w)(\fq))$, is the ideal generated by homogeneous polynomials vanishing on $\PP(w)(\fq)$.

\begin{thm}[{\cite[Prop.~5.6]{mesutComputingVanishing}}]\label{thm:binomial}
    $\calI(\mathbb{P}(w)(\fq))$ is binomial.
\end{thm}

Binomials in $\calI(\mathbb{P}(w)(\fq))$ are easy to characterize \cite[Theorem 3.7]{mesutComputingVanishing}:
\[ x_0^{a_0}\cdots x_m^{a_m} - x_0^{b_0}\cdots x_m^{b_m} \in \calI(\mathbb{P}(w)(\fq))\]
if and only if $\set{i\in\set{0,\dots,m} : a_i = 0} = \set{i\in\set{0,\dots,m} : b_i = 0}$ and $(q-1)$ divides $b_i-a_i$ for every $i \in \set{0,\dots,m}$. By the previous characterization, for $0 \leq i < j \leq m$, the binomial
\begin{equation}\label{eq:Bij}
B_{i,j}:=x_ix_j\left(x_j^{w_i(q-1)}-x_i^{w_j(q-1)}\right)
\end{equation}
belongs to  
$\calI(\mathbb{P}(w)(\fq))$ with 
\begin{equation}\label{eq:deg_Bij}
    \deg(B_{i,j})=w_iw_j(q-1)+w_i+w_j.
\end{equation}
When $w=(1,\dots,1)$, or more generally when $w_0=\dots=w_{m-2}=1$, the ideal $\calI(\mathbb{P}(w)(\fq))$ is generated by the set of binomials $\set{B_{i,j} : 0 \leq i < j \leq m}$  \cite[Theorem~5.7]{mesutComputingVanishing}. 
For more general weights, this may not be the case anymore, and finding a minimal generating set of binomials for $\calI(\mathbb{P}(w)(\fq))$ is more intricate, as it depends on the relations between the weights (see \cite[Proposition~5.10]{mesutComputingVanishing} for $m=2$ with $w_0 \neq 1$).

We consider the degree lexicographic order with $x_0 < x_1 < \dots< x_m$. Let
\begin{equation}
    \Mon:=\set{x_0^{a_0}\cdots x_m^{a_m} \in \M : \forall f \in  \calI(\mathbb{P}(w)(\fq)) \text{ homogeneous}, \: \ini(f) \nmid x_0^{a_0}\cdots x_m^{a_m} }
\end{equation}
and $\Mon_d:=\Mon \cap \M_d$. We can write
\[\Mon=\bigsqcup_{d \geq 0} \Mon_d.\]
Then the quotient ring $\fq[x_0,\dots,x_m]^w /\calI(\mathbb{P}(w)(\fq))$ is generated by $\Mon$ as an $\fq$-vector space \cite[Thm. 15.3]{eisenbud}. For any $d \geq 0$ and $0\leq i \leq m$, we define
\begin{equation}\label{eq:Mondi}
    \Moni{d}{i}:=\set{x_0^{a_0}\cdots x_m^{a_m} \in \Mon_d  : \forall t < i, \: a_t = 0 \text{ and } a_i\geq 1},
\end{equation}
which makes a partition of $\Mon_d=\bigsqcup_{i=0}^m \Moni{d}{i}$. Similarly, we set
\begin{equation}
    \Moni{}{i}:=\bigsqcup_{d \geq 0} \Moni{d}{i}.
\end{equation}

\begin{rem}
The notation we are considering is analogous to the one used in \cite{projectivefootprint} for projective reductions, but the ordering of the variables is reversed.
\end{rem}

As mentioned in the introduction, we assume that 
\[1=w_0 \leq w_1=\dots=w_\ell < w_{\ell+1} \leq \cdots \leq w_m\]
with $1\leq \ell \leq m-1$, which implies $w_m\geq 2$ (recall that for the case $1=w_0=\cdots=w_m$, Theorem \ref{t:theorem_conjecture_intro} is Serre's bound). We give now some basic facts about the sets of monomials we have just introduced and the regularity set of $\PP(w)(\fq)$.

\begin{lem}\label{lem:form_monomials}
Let $x_{j_0}^{a_{j_0}}\cdots x_m^{a_m} \in \Moni{}{j_0}$.
If $w_{j_0}=1$, then $0\leq a_j \leq q-1$ for every $j \in \set{j_0+1,\dots, m}$.    
\end{lem}
\begin{proof}
    Assume that $a_j \geq q$ for some $j \geq j_0+1$.
    Set $\rho=\Floorfrac{a_j-1}{q-1}\geq 1$.
    As the binomial $B_{j_0,j}=x_{j_0}x_j\left(x_j^{q-1}-x_{j_0}^{w_j(q-1)}\right)$ lies in $\calI(\mathbb{P}(w)(\fq))$, we have 
    \[x_{j_0}^{a_{j_0}}\cdots x_m^{a_m} - x_{j_0}^{a_{j_0}+\rho w_j(q-1)}\cdots x_j^{a_j-\rho(q-1)} \cdots x_m^{a_m}\in \calI(\PP(w)(\fq)).\]
    One can check that $0 \leq a_j-\rho(q-1) \leq q-1$. Repeating this for every $j \in \set{j_0+1,\dots,m}$ gives the result.
\end{proof}

\begin{defn}\label{def:reg}
The {\it regularity set} of $\PP(w)(\fq)$ is defined as 
    \begin{equation*}\label{eq:def_reg}
    \reg (\PP(w)(\fq)):=\{d \in \mathbb{N}_{\geq 1}: H_{\PP(w)(\fq)}(d)=\abs{\PP(w)(\fq)}=p_m\},
    \end{equation*}
where $H_{\PP(w)(\fq)}(d):=\dim\fq[x_0,\dots,x_m]^w_d/\calI(\PP(w)(\fq))_d=\abs{\Mon_d}$ denotes the value of the Hilbert function of $\fq[x_0,\dots,x_m]^w/\calI(\PP(w)(\fq))$ at degree $d$. 
\end{defn}

We do not know $\reg (\PP(w)(\fq))$ in general, although it has been computed for the particular case of $\PP(1,w_1,w_2)$ in \cite{caki_sahin_nardi25}. Nevertheless, for our purposes it is sufficient to know that this set is non-empty and infinite. Indeed, it is non-empty by \cite[Prop. 4.4]{maclagan_multigraded_regularity} and the characterization given in \cite[Prop. 6.7]{maclagan_multigraded_regularity} (the introduction of \cite{maclagan_multigraded_regularity} explains the results for projective toric varieties, and, in particular, Example 2.1 deals with $\PP(w)$).
That $\reg(\PP(w)(\fq))$ is infinite can be deduced from the fact that, if $\tilde{d}\in \reg(\PP(w)(\fq))$, then $p_m=H_{\PP(w)(\fq)}(\tilde{d})\leq H_{\PP(w)(\fq)}(\tilde{d}+\lcm(w)(q-1))\leq p_m$. Indeed, the last inequality follows from \cite[Thm. 3.7]{sahin_soprunov_multigraded_hilbert_functions}, and the first from the next lemma.

\begin{lem}
Let $d\geq 1$. Then $H_{\PP(w)(\fq)}(d)\leq H_{\PP(w)(\fq)}(d+\lcm(w)(q-1))$.
\end{lem}
\begin{proof}
Recall that $H_{\PP(w)(\fq)}(d)=\abs{\Mon_{d}}$. Thus, it is enough to prove that the set
\[\bigcup_{i=0}^m x_i^{(q-1)\lcm(w)/w_i} \cdot \Moni{d}{i}\subset  \fq[x_0,\dots,x_m]^w_{d+(q-1)\lcm(w)}\]
of cardinality $\abs{\Mon_{d}}$ is linearly independent modulo $\calI(\PP(w)(\fq))$. Arguing by contradiction, assume that it is linearly dependent. This is equivalent to having a linear combination which vanishes at all the points of $\PP(w)(\fq)$. This, in turn, would give a linear combination of monomials in $\bigcup_{i=0}^m \Moni{d}{i}$ vanishing at all the points of $\PP(w)(\fq)$. Indeed, given $x_i^{a_i}\cdots x_m^{a_m} \in  \Moni{d}{i}$ (i.e., $a_i\geq 1$), we have $x_i^{a_i}\cdots x_m^{a_m}(P)=(x_i^{(q-1)\lcm(w)/w_i} x_i^{a_i}\cdots x_m^{a_m})(P)$ for all $P\in \PP(w)(\fq)$. But this is a contradiction, since $\Mon_d$ is linearly independent in $\calI(\PP(w)(\fq))$. 
\end{proof}

\begin{lem}\label{lem:Mondi_d_reg}
We have that $\tilde{d} \in \reg(\mathbb{P}(w)(\fq))$ if and only if $\size{\Moni{\tilde{d}}{i}}=q^{m-i}$ for every $0\leq i \leq m$.
\end{lem}
\begin{proof}
On the one hand, the definition of the regularity set (Definition~\ref{def:reg}) means that
\[
\abs{\Mon_{\tilde{d}}}=\dim\left(\fq[x_0,\dots,x_m]^w_d/\calI(\PP(w)(\fq))_d\right)=p_m=\sum_{i=0}^m \size{\Moni{\tilde{d}}{i}}.
\]
On the other hand, the definition of $\Moni{\tilde{d}}{i}$ (Equation~\eqref{eq:Mondi}) entails that $\size{\Moni{\tilde{d}}{i}}\leq q^{m-i}$, which gives the result.
\end{proof}

We now introduce one of the main tools for our proof.

\begin{defn}\label{d:FB}
    Let $d \geq 1$ and $\tilde{d} \in \reg(\mathbb{P}(w)(\fq))$ such that $\tilde{d}\gg d$. Let $x_0^{a_0}\cdots x_m^{a_m} \in \Mon_d$ and $i \in \set{0,\dots,m}$.
    We define the \emph{footprint bound} associated to the monomial $x_0^{a_0}\cdots x_m^{a_m}$ as
\[\FB(x_0^{a_0}\cdots x_m^{a_m}):=\size{\set{x_0^{b_0} \dots x_m^{b_m} \in \Mon_{\tilde{d}} : x_0^{a_0}\cdots x_m^{a_m} \text{ divides } x_0^{b_0} \dots x_m^{b_m} }}.
    \]
    We also define, for every $i \in \set{0,\dots,m}$,    
    \[\FBi{i}(x_0^{a_0}\cdots x_m^{a_m}):=\size{\set{x_i^{b_i} \dots x_m^{b_m} \in \Moni{\tilde{d}}{i} : x_0^{a_0}\cdots x_m^{a_m} \text{ divides } x_i^{b_i} \dots x_m^{b_m} }}.
    \]
    \end{defn}

The previous definition is motivated by the following result, which is a consequence of \cite[Thm. 5.2]{nardi_projective_toric} and can also be found in \cite[Lem. 4.2]{caki_sahin_nardi25}. 

\begin{thm}
Let $f\in \fq[x_0,\dots,x_m]^w_d\setminus I_d(\PP(w))$. Then
$$
\abs{V_{\PP(w)}(f)(\fq)}\leq p_m- \FB(\ini(f)).
$$
\end{thm}

The following elementary properties of $\FBi{i}$ will be used throughout the next section.

\begin{lem}\label{lem:FBe}
Let $d \geq 0$ and fix $0 \leq j_0 \leq m$. Let $x_{j_0}^{a_{j_0}}\cdots x_m^{a_m} \in \Moni{d}{j_0}$. 
    \begin{enumerate}[label=(\roman*)]
        \item For every $i \geq j_0+1$, we have $\FBi{i}(x_{j_0}^{a_{j_0}}\cdots x_m^{a_m})=0$.
        \item\label{it:sum_FB} $\FB(x_{j_0}^{a_{j_0}}\cdots x_m^{a_m})=\sum_{i=0}^{j_0} \FBi{i}(x_{j_0}^{a_{j_0}}\cdots x_m^{a_m})$. 
        \item\label{it:FBe}
        Assume that $d \leq w_1q$. For every $i  \leq j_0$ such that $w_0=\dots=w_i=1$, we have
    \[\FBi{i}(x_{j_0}^{a_{j_0}}\cdots x_m^{a_m}) = \prod_{j=i+1}^m(q-a_j).\]
     \end{enumerate}
\end{lem}
\begin{proof}
The first two properties directly follow from the definition of $\FBi{i}(x_{j_0}^{a_{j_0}}\cdots x_m^{a_m})$.
Let us prove \ref{it:FBe}. 
First notice that the hypothesis $d \leq w_1 q$ ensures that $0 \leq a_1,\dots,a_m \leq q-1$, which means $q-a_j \geq 1$ for every $j \geq 1$.

Now take $i \in \set{0,\dots,m-1}$ such that $w_0=\dots=w_i=1$. Since $\tilde{d} \in\reg(\mathbb{P}(w)(\fq))$, Lemmas \ref{lem:form_monomials} and \ref{lem:Mondi_d_reg} imply that
\[\Moni{\tilde{d}}{i}=\set{x_i^{b_i} \dots x_m^{b_m} : 0 \leq b_j \leq q-1 \text{ for }j \geq i+1}. \]
Moreover, by definition of the ordering and the fact that $w_i=1$, we have that $b_i=\tilde{d}-\sum_{j=i+1}^m w_ja_j$. For $\tilde{d}$ large enough, it is always true that $a_i \leq b_i$. As a result, the monomial $x_{j_0}^{a_{j_0}}\cdots x_m^{a_m}$ divides a monomial $x_i^{b_i} \dots x_m^{b_m}\in \Moni{\tilde{d}}{i}$ if and only if $0\leq a_j \leq b_j \leq q-1$ for every $j \geq i+1$, hence the formula for $\FBi{i}(x_{j_0}^{a_{j_0}}\cdots x_m^{a_m})$.
\end{proof}

\section{Proof of the main theorem}\label{s:proof}

This section is dedicated to the proof of the Theorem~\ref{t:theorem_conjecture_intro}. We start by computing the footprint bound associated to monomials involving at least one variable among $x_0, x_{\ell+1},\dots,x_m$ (Lemmas~\ref{lem:FB_x0} and \ref{lem:ws>1}). We deal with monomials that depend only on the $\ell$ variables $x_1,\dots,x_\ell$ separately (Lemma~\ref{lem:var_weights=w_1}), as their corresponding footprint bound may be smaller than the value expected by Theorem~\ref{t:theorem_conjecture_intro} (see Example~\ref{ex:bad_fb}).

\begin{lem}\label{lem:FB_x0}
Assume $w_0=1$. Fix $1 \leq d \leq w_1q$.
    Let $x_0^{a_0}\cdots x_m^{a_m} \in \Moni{d}{0}$ (i.e., $a_0\geq 1$). 
    \begin{itemize}
        \item If $w_1 \mid d$, then 
        \[\FB(x_0^{a_0}\cdots x_m^{a_m}) \geq \FB\left(x_0^{w_1}x_1^{d/w_1-1}\right)=q^{m-1}\left(q-\frac{d}{w_1}+1\right).\]
         \item If $w_1 \nmid d$, then 
         \[\FB(x_0^{a_0}\cdots x_m^{a_m}) \geq \FB\left(x_0^{a^\star_0}x_1^{\Floorfrac{d-1}{w_1}}\right)=q^{m-1}\left(q-\Floorfrac{d-1}{w_1}\right),\]
         where $a^\star_0$ is the positive integer such that $d=a^\star_0+\Floorfrac{d-1}{w_1}w_1$ and $1 \leq a^\star_0 \leq w_1-1$.
    \end{itemize}
\end{lem}

\begin{rem}\label{rem:same_formulae}
    Note that when $w_1 \mid d$, the second item of Lemma~\ref{lem:FB_x0} matches the first one when allowing $a^\star_0=w_1$, since we have $\Floorfrac{d-1}{w_1}=\frac{d}{w_1}-1$.
\end{rem}

\begin{proof}[Proof of Lemma~\ref{lem:FB_x0}]
We prove the second formula, which is more general (see Remark \ref{rem:same_formulae}). To do that, we generalize the ideas of the proof of \cite[Lemma~1]{geilSecondWeightGRM} (see also \cite[Example~5.19]{tohaneanuCommutativeAlgebraforCodingTheory}).
    
    The hypotheses $d \leq w_1 q$ and $a_0 \geq 1$ ensure that $0 \leq a_1,\dots,a_m \leq q-1$. By Lemma \ref{lem:FBe} \ref{it:sum_FB} and \ref{it:FBe}, we have 
    \[\FB(x_0^{a_0}\cdots x_m^{a_m})= \FBi{0}(x_0^{a_0}\cdots x_m^{a_m})=\prod_{j=1}^m(q-a_j).\]
    As $w_0=1$, the exponent $a_0$ is completely determined by the other $m$ ones: $a_0=d-(w_1a_1 + \dots +w_ma_m)$. Since we require $a_0 \geq 1$, we aim to minimize the $m$-variate function $F_0$ defined by
    \[F_0(a_1,\dots,a_m):=\prod_{j=1}^m(q-a_j)\]
    over the domain 
    \[\mathcal{D}=\set{ (a_1,\dots,a_m) \in \set{0,\dots,q-1}^m : w_1a_1+\dots+w_ma_m \leq d-1}.\]

    Suppose $(a_1,\dots,a_m)\in \mathcal{D}$ is a tuple for which the minimum is attained. Then for every $j\in \set{1,\dots,m}$, the tuple $(a_1,\dots,a_j+1,\dots,a_m)$ does not belong to $\mathcal{D}$ because
        \[F_0(a_1,\dots,a_j+1,\dots,a_m) < F_0(a_1,\dots,a_j,\dots,a_m).\]
Therefore $ w_1a_1+\dots+w_ma_m \geq d-\min\set{w_j : j \geq 1}=d-w_1$, which means that $a_0=d-(w_1a_1+\dots+w_ma_m)\leq w_1$.

Let us prove that we can assume without loss of generality that $a_j=0$ for every $j \geq \ell +1$ (i.e., $w_j > w_1$). Assume by contradiction that there exists an index $j$ such that $a_j \geq 1$ and $w_j > w_1$. Then $a_jw_j \geq w_1+1$. Write the Euclidean division of $a_jw_j-(w_1+1)$ by $w_1$ as 
\[a_jw_j-(w_1+1)=\beta_1w_1+\beta_0,\] 
 with $\beta_1 \geq 0$ and $0\leq \beta_0 \leq w_1-1$.
Set $\alpha_1=\beta_1+1 \geq 1$ and $\alpha_0=\beta_0+1\leq w_1$. Then we can write $a_jw_j=\alpha_1w_1+\alpha_0$. In particular, when $w_1=1$, we have $\alpha_0=1$ and $\alpha_1=a_jw_j-1$.

Then $(a_1+\alpha_1,a_2,\dots,a_{j-1},0,a_{j+1},\dots,a_m) \in \mathcal{D}$. Moreover we have
\[a_j \leq \frac{w_1(\alpha_1+1)}{w_j}< \frac{w_j(\alpha_1+1)}{w_j}=\alpha_1+1,\]
hence $a_j \leq \alpha_1$. Then 
\begin{align*}
 F_0(a_1,\dots,a_j,\dots,a_m)-F_0(a_1+\alpha_1,&a_2,\dots,a_{j-1},0,a_{j+1},\dots,a_m)\\
 &= \left( (q-a_1)(q-a_j) - (q-a_1-\alpha_1)q\right) \prod_{i \neq 1,j} (q-a_i)\\
 &=\left(a_1a_j + q(\alpha_1-a_j)\right)\prod_{i \neq 1,j} (q-a_i)\geq 0,
\end{align*}

which implies that $F_0(a_1,\dots,a_j,\dots,a_m) \geq F_0(a_1+\alpha_1,a_2,\dots,a_{j-1},0,a_{j+1},\dots,a_m)$.

Recall that $w_1=\dots=w_\ell$. We can thus now assume that $(a_1,\dots,a_m)=(a_1,\dots,a_\ell,0,\dots,0)$.
\begin{itemize}
    \item If $\ell=1$ (i.e., $w_2>w_1$), then $(a_1,0,\dots,0)$ gives the minimum of the function $F_0$ if and only if $a_1$ is maximal with $w_1a_1 \leq d-1$, that is, $a_1=\Floorfrac{d-1}{w_1}$.
    \item Assume $\ell \geq 2$. For every $j \in \set{2,\dots,\ell}$ the tuple $(a_1+a_j,a_2,\dots,0,\dots,a_\ell,0,\dots,0)$ belongs to $\mathcal{D}$ and
    \begin{align*}
 F_0(a_1,a_2,\dots)-F_0(a_1+a_j,&a_2,\dots,0,\dots,a_\ell,0,\dots,0)  \\
 &= \left( (q-a_1)(q-a_j) - (q-a_1-a_j)q)\right) \prod_{i\neq 1,j} (q-a_i)\\
 &=a_1a_j\prod_{i\neq 1,j} (q-a_i) \geq 0
\end{align*}
We can then assume that $a_2=\dots=a_\ell=0$. As in the previous case, $a_1$ has to be maximal, i.e., $a_1=\Floorfrac{d-1}{w_1}$.
\end{itemize}
In both cases, we have
  \begin{align*}
      \min\set{\FB(x_0^{a_0}\cdots x_m^{a_m}) : x_0^{a_0}\cdots x_m^{a_m} \in \Moni{d}{0}} &= F_0\left(\Floorfrac{d-1}{w_1},0,\dots,0\right)\\
      &=\FB\left(x_0^{a^\star_0}x_1^{\Floorfrac{d-1}{w_1}}\right)\\
      &=q^{m-1}\left(q-\Floorfrac{d-1}{w_1}\right),
  \end{align*}
  where $a^\star_0=d-\Floorfrac{d-1}{w_1}w_1 \geq 1$.
\end{proof}

\begin{lem}\label{lem:ws>1}
Let $x_{1}^{a_1}\cdots x_s^{a_s}\in \Mon \setminus \Moni{}{0}$ (i.e., $a_0=0$) with $a_{s}>0$ and $w_{s}>w_1$. Then
\begin{equation}\label{eq:fbwj=1}
\FB(x_{1}^{a_1}\cdots x_s^{a_s})\geq \FB(x_0^{\alpha_0}x_1^{a_1+\alpha_1} \cdots x_{s-1}^{a_{s-1}})
\end{equation}
where $\alpha_0$ and $\alpha_1$ are the unique integers such that $a_sw_s=\alpha_0+\alpha_1w_1$ and $1 \leq \alpha_0 \leq w_1$, $\alpha_1 \geq 1$. 
\end{lem}

\begin{proof}
We suppose that $a_s \geq 1$ and $w_s>w_1$. As in the second part of the proof of Lemma \ref{lem:FB_x0}, we can prove that
\[
\begin{aligned}
\FB(x_1^{a_1}\cdots x_s^{a_s}) \geq \FBi{0}(x_1^{a_1}\cdots x_s^{a_s}) =(q-a_1)\prod_{i=2}^m(q-a_i)&\geq (q-a_1-\alpha_1)\prod_{i=2}^m(q-a_i) \\
&= \FB(x_0^{\alpha_0}x_1^{a_1+\alpha_1} \cdots x_{s-1}^{a_{s-1}}).
\end{aligned}
\]
\end{proof}

With the previous lemmas, we can find candidates for the monomials with lowest footprint bound. However, there are some cases in which the bound for the number of zeroes given by the footprint bound is lower than Conjecture \ref{sudhir conjecture}. This happens, for example, when $w_1=w_2\mid d$, since we may have $\FB\left(x_2^{d/w_2}\right)<\FB\left(x_1^{d/w_1}\right)=\FB\left(x_0x_1^{d/w_1-1}\right)=q^{m-1}(q-\frac{d}{w_1}+1)$.
\begin{ex}\label{ex:bad_fb}
    If $w_1=w_2=1$ and $d>1$, Lemma~\ref{lem:FBe}~\ref{it:FBe} gives
    \[\FB\left(x_2^{d/w_2}\right)=q^{m-2}\left((q-d)(q+1)+1\right)<q^{m-1}\left(q-d+1\right)=\FB\left(x_1^{d/w_1}\right).\]
\end{ex}
In the next result, we use different techniques to bound the number of zeroes for these cases.

\begin{lem}\label{lem:var_weights=w_1}
Assume that $1 =w_0 \leq w_1=\dots=w_\ell$ and $w_1 \mid d$. Let $f \in \Span\Mon_d$ such that its initial monomial $\ini(f)=x_1^{a_1}\cdots x_\ell^{a_\ell}$ is only divisible by $x_1$, \dots, $x_\ell$ (i.e., $a_0=a_{\ell+1}=\dots=a_m=0$). Then
\[
\abs{V_{\mathbb{P}(w)}(f)(\fq)}\leq \frac{d}{w_1}q^{m-1}+p_{m-2}.
\]
\end{lem}

\begin{proof}
First note that a monomial of the form $x_1^{a_1}\cdots x_\ell^{a_\ell}$ lies in $\M_d$ if and only if $w_1$ divides $d$. By the definition of the ordering, $f$ has no monomial divisible by some variable $x_j$ for $j \geq \ell+1$, and $f$ depends only on the first $\ell+1$ variables. Moreover, the exponent of $x_0$ in every monomial of $f$ is divisible by $w_1$.

We now prove that
\begin{equation}\label{eq:tildef}
    \abs{V_{\mathbb{P}(w)}(f)(\fq)}\leq\abs{V_{\mathbb{P}^\ell}\left(\tilde{f}\right)(\fq)}q^{m-\ell}+p_{m-\ell-1},
\end{equation}
where $\tilde{f} \in \fq[y_0,\dots,y_\ell]_{d/w_1}^{(1,\dots,1)}$ is the homogeneous polynomial of degree $d/w_1$ that satisfies $\tilde{f}(x_0^{w_1},x_1,\dots,x_\ell)=f(x_0,\dots,x_\ell,b_{\ell+1},\dots,b_m)$ for every $(b_{\ell+1},\dots,b_m) \in \fq^{m-\ell}$. 
Let $P=(P_0:\dots:P_m) \in V_{\mathbb{P}(w)}(f)(\fq)$.
\begin{itemize}
    \item If $(P_0,\dots,P_\ell)\neq(0,\dots,0)$, then Lemma \ref{l:delorme} implies that $(P_0^{w_1}:P_1:\dots:P_m) \in \PP^{\ell}(\fq)$. Moreover, it is a zero of $\tilde{f}$, i.e., $(P_0^{w_1}:P_1:\dots:P_m) \in V_{\mathbb{P}^\ell}\!\left(\tilde{f}\right)(\fq)$. 
    
    Conversely, Lemma \ref{l:delorme} ensures that any $\fq$-point of $\mathbb{P}^\ell$ can be written in the form $(P_0^{w_1}:P_1:\dots:P_m)$ with  $(P_0:P_1:\dots:P_m)\in\PP(1,w_1,\dots,w_1)$. Then for any $(P_0^{w_1}:\dots:P_m) \in V_{\mathbb{P}^\ell}\!\left(\tilde{f}\right)(\fq)$ and any $(Q_{\ell+1},\dots,Q_m) \in \fq^{m-\ell}$, the point $(P_0:\dots:P_\ell:Q_{\ell+1}:\dots:Q_m)$ belongs to $V_{\mathbb{P}(w)}(f)(\fq)$. However, it is worth noting that two different $(Q_{\ell+1},\dots,Q_m)\neq(Q'_{\ell+1},\dots,Q'_m)$ in $\fq^{m-\ell}$ may define the same point 
    $(P_0:\dots:P_\ell:Q_{\ell+1}:\dots:Q_m)=(P_0:\dots:P_\ell:Q'_{\ell+1}:\dots:Q'_m)$ in $\PP(w)(\fq)$ when $w_1 \geq 2$.
     Therefore, each element of $V_{\mathbb{P}^\ell}\!\left(\tilde{f}\right)(\fq)$ gives rise to \emph{at most} $q^{m-\ell}$ elements of the form $(P_0:\dots:P_\ell:Q_{\ell+1}:\dots:Q_m)$ with $(Q_{\ell+1},\dots,Q_m) \in \fq^{m-\ell}$. 
    
    \item Otherwise, we have $(P_0,\dots,P_\ell)=(0,\dots,0)$ and for every $(Q_{\ell+1},\dots,Q_m) \in \PP(w_{\ell+1},\dots,w_m)(\fq)$, the point $(0:\cdots :0:Q_{\ell+1}:\cdots:Q_m)$ lies in $V_{\mathbb{P}(w)}(f)(\fq)$.
\end{itemize}
This leads to
\begin{align*}
     \abs{V_{\mathbb{P}(w)}(f)(\fq)}&= \abs{V_{\mathbb{P}(w)}(f)(\fq) \setminus \bigcap_{i=0}^\ell V_{\mathbb{P}(w)}(x_i)(\fq)} + \abs{V_{\mathbb{P}(w)}(f)(\fq) \cap \bigcap_{i=0}^\ell V_{\mathbb{P}(w)}(x_i)(\fq)}\\
&\leq\abs{V_{\mathbb{P}^\ell}\!\left(\tilde{f}\right)(\fq)}q^{m-\ell}+\abs{\PP(w_{\ell+1},\dots,w_m)(\fq)},
\end{align*}
which proves \eqref{eq:tildef}. Serre's bound \cite{serreZeroesProjectiveSpace} on the polynomial $\tilde{f}$ over $\PP^\ell(\fq)$ gives that
\begin{equation}\label{eq:Serre}
    \abs{V_{\mathbb{P}^\ell}\!\left(\tilde{f}\right)(\fq)} \leq \frac{d}{w_1}q^{\ell-1}+p_{\ell-2}.
\end{equation}
Injecting the last inequality in Equation \eqref{eq:tildef} gives the expected bound.
\end{proof}

We have now gathered all the ingredients to prove the main result of the paper.

\begin{thm}\label{t:theorem_conjecture}
Assume that $w_0=1$. Then 
\[
e_q(d;w_0,w_1,\dots,w_m)= \min\set{p_m,\left(\Floorfrac{d-1}{w_1}+1\right)q^{m-1}+p_{m-2}}.
\]
\end{thm}

\begin{proof}
We start by proving the inequality $e_q(d;w_0,w_1,\dots,w_m)\geq \min\{p_m,dq^{m-1}+p_{m-2}\}$. As in the proof of \cite[Thm. 5.1]{aubryperretWPRM}, for $d\geq w_1q+1$ the polynomial
\[
x_0^{d-w_1q-1}B_{0,1}=x_0^{d-w_1q-1}(x_0x_1^q-x_0^{w_1(q-1)+1}x_1)
\]
is homogeneous of degree $d$ and has $p_m$ rational zeroes. If $d \leq qw_1$, set $d_1=\Floorfrac{d-1}{w_1}\leq q-1$ and write $d=a_0^\star+d_1w_1$ with $1 \leq a_0^\star \leq w_1$ as in Lemma \ref{lem:FB_x0}. Consider the polynomial
\begin{equation}\label{eq:min_pol}
f=x_0^{a_0^\star}\prod_{i=1}^{d_1}(x_1-\alpha_ix_0^{w_1}),    
\end{equation}
where $\alpha_1,\dots,\alpha_{d_1}$ are distinct nonzero elements of $\fq$. Then $f$ is a homogeneous polynomial of degree $d$ with $(d_1+1)q^{m-1}+p_{m-2}$ zeroes, and we have thus obtained the desired inequality. 

\medskip

Now we prove the inequality $e_q(d;w_0,w_1,\dots,w_m)\leq \min\{p_m,dq^{m-1}+p_{m-2}\}$. It is enough to consider the case $d\leq w_1q$, as $e_q(d;w_0,w_1,\dots,w_m)\leq p_m$ and $\left(\Floorfrac{d-1}{w_1}+1\right)q^{m-1}+p_{m-2}\geq p_m$ for $d\geq w_1q+1$.

Let $f\in \fq[x_0,\dots,x_m]^w_d$, and let $\ini(f)=x_{j_0}^{a_{j_0}}\cdots x_s^{a_s}$ with $a_{j_0} \geq 1$. Without loss of generality, we may assume that $\ini(f)\in\Moni{d}{j_0}$.
\begin{itemize}
    \item If $j_0=0$, then Lemma \ref{lem:FB_x0} ensures that
    \[
    \size{V_{\PP(w)}(f)(\fq)} \leq p_m-\FB(\ini(f)) \leq  \left(\Floorfrac{d-1}{w_1}+1\right)q^{m-1}+p_{m-2}.
    \]
    \item Otherwise, we have to distinguish two cases:
    \begin{itemize}
        \item either $a_j\geq 1$ for some $j \geq \ell+1$ (hence $w_j > w_1$), in which case Lemma \ref{lem:ws>1} implies that
        \[\FB(\ini(f)) \geq \min_{\nu \in \Moni{d}{0}} \FB(\nu) = q^{m-1}\left(q-\Floorfrac{d-1}{w_1}+1\right),\]
        thus giving the expected upper bound on $\size{V_{\PP(w)}(f)(\fq)}$;
        \item or $\ini(f)=x_1^{a_1}\cdots x_\ell^{a_\ell}$ with $w_1 \mid d$ as in Lemma \ref{lem:var_weights=w_1}. In this case, we have $\Floorfrac{d-1}{w_1}=\frac{d}{w_1}-1$ and Lemma \ref{lem:var_weights=w_1} gives exactly the desired result.
    \end{itemize}
\end{itemize}   
\end{proof}

Serre \cite{serreZeroesProjectiveSpace} noted that in the classical projective space $(w_0=\dots=w_m=1)$, the bound is achieved exactly by the unions of hyperplanes that share a common linear variety of dimension $m-2$. Here, when $w_0=1$, the variety associated to the polynomial $f$ achieving the bound given in Equation \eqref{eq:min_pol} is the union of the hyperplane $x_0=0$ and some hypersurfaces of the form $x_1=\alpha x_0^{w_1}$ for $\alpha\in\fq$. These hypersurfaces may be considered as hyperplanes, since they contain $p_{m-1}$ $\fq$-points. They all contain the linear subvariety defined by $x_0=x_1=0$, which contains $p_{m-2}$ $\fq$-points.

\section{Generalizations to $w_0\geq 2$}\label{s:w0>1}

When $w_0 \geq 2$, there are still some cases where Theorem \ref{t:theorem_conjecture} can be applicable. If $\gcd(w)=\gamma>1$, then any degree $d$ such that $\fq[x_0,\dots,x_m]^w_d\neq\set{0}$ lies in the semigroup $\gen{w_0,\dots,w_m}_\NN$. By Lemma \ref{l:weightswithgcd}, for any nonzero $f \in \fq[x_0,\dots,x_m]^w_d$, we have 
\begin{equation}\label{eq:nb_pts_with_gcd}
    \abs{V_{\PP(w)}(f)(\fq)}=\abs{V_{\PP(w_0/\gamma,\dots,w_m/\gamma)}(f)(\fq)},
\end{equation}
where the polynomial $f$ on the right-hand side is considered in $\fq[x_0,\dots,x_m]^{(w_0/\gamma,\dots,w_m/\gamma)}_{d/\gamma}$. Then
\begin{equation}
e_q(d;w_0,w_1,\dots,w_m)=e_q(d/\gamma;w_0/\gamma,w_1/\gamma,\dots,w_m/\gamma).
\end{equation}
Thus, if $w_0=\gamma$, we can apply Theorem \ref{t:theorem_conjecture}. In what follows, we can therefore assume that $\gcd(w)=1$. A similar argument can be made with Delorme's reduction (Lemma \ref{l:delorme}), but then $d$ has to be a multiple of the greatest common divisor of the $m$ weights that have a common factor. 

Assume now that we cannot use Lemma \ref{l:weightswithgcd} or Lemma \ref{l:delorme} to reduce to the case with $w_0=1$. A natural path towards generalizing Theorem \ref{t:theorem_conjecture} could be expecting that when $d$ lies in the semigroup generated by $w_0$ and $w_1$, taking a polynomial $f$ giving a maximum number of zeroes in $\PP(w_0,w_1)(\fq)$ would define a variety in $\PP(w)$ with maximum number of $\fq$-zeroes. We will show that this is not the case in general, but first we study the case $m=1$ in detail.

\subsection{Case $m=1$}

Let us study the maximum number of $\fq$-zeroes of bivariate polynomials in $\fq[x_0,x_1]^w_d$ and the polynomials that achieve it. Arguing as above, we assume that $w_0$ and $w_1$ are coprime, and, although $\PP(w_0,w_1)$ is isomorphic to the straight projective line $\PP^1$, Delorme's reduction (Lemma \ref{l:delorme}) only enables us to state that
\begin{equation}
e_q(d;w_0,w_1)=e_q(d/(w_0w_1),1,1)=\frac{d}{w_0w_1}
\end{equation}
if the degree $d$ is divisible by $w_0w_1$. For general degrees, the corresponding isomorphism on the sheaves $\mathcal{O}_{\PP(w)}(d)$ is more involved (e.g., see \cite[Proposition~3C.7]{Beltrametti}). We reformulate this in terms of the denumerant $\den(d;w_0,w_1)$ to get the following bound. 

\begin{prop}\label{prop:bound_m=1}
Let $(w_0,w_1) \in \NN^2$ such that $\gcd(w_0,w_1)=1$ and $d \in \gen{w_0,w_1}_\NN$. Set 
\begin{equation}\label{eq:delta}
    \delta=\den(d;w_0,w_1)-1.
\end{equation} 
Then for every nonzero $f \in \fq[x_0,x_1]^w_d$, we have \[\abs{V_{\PP(w_0,w_1)}(f)(\fq)} \leq \begin{cases}
    \delta & \text{if both }w_0 \text{ and } w_1 \text{ divide } d,\\
    \delta+1 & \text{if either }w_0\text{ or } w_1 \text{ divides } d,\\
    \delta+2 & \text{if neither }w_0 \text{ nor } w_1 \text{ divides } d.\\
\end{cases}\]
\end{prop}

\begin{proof} 
If $w_0w_1 \mid d$, then it is clear that $\delta=\frac{d}{w_0w_1}$ and the result follows from Equation~\eqref{eq:nb_pts_with_gcd}. Let us assume that $w_0w_1 \nmid d$. Since $\gcd(w_0,w_1)=1$, one can uniquely write $d=w_0 i^*_0 + w_1 i^*_1$ with $0 \leq i_0^*$ and $0 \leq i^*_1 \leq w_0-1$, and
\[\fq[x_0,x_1]^w_d=\Span \set{x_0^{i^*_0-\ell w_1}x_1^{i^*_1 + \ell w_0}: \: 0 \leq \ell \leq \delta  }. \]
Then any $f \in \fq[x_0,x_1]^w_d$ can be written as
\[f=x_0^{i^*_0}x_1^{i^*_1} g\left(\frac{x_1^{w_0}}{x_0^{w_1}}\right)^\ell,\]
where $g$ is a univariate polynomial of degree at most $\delta$. 
Moreover, $w_0$ divides $d$ if and only if $i^*_1= 0$. Then $f$ is necessarily divisible by $x_1$ if $w_0$ does not divide $d$. Similarly, $w_1$ divides $d$ if and only if $i^*_0=w_1\delta$. Indeed, we can write the Euclidean division of $d/w_1$ by $w_0$ as
\[
d=(hw_0+r)w_1=hw_0w_1+rw_1,
\]
where $r\leq w_0-1$. By uniqueness of the couple $(i^*_0,i^*_1)$, we get $i^*_1=r$ and $i^*_0=hw_1$, but $h=\Floorfrac{d/w_1}{w_0}=\Floorfrac{d}{w_0w_1}=\delta$ by the definition of $\delta$ and the fact that $\den(d;w_0,w_1)=\den(d/w_1;w_0,1)$. Moreover, $f$ is necessarily divisible by $x_0$ if $w_1$ does not divide $d$.

On the affine open patch $x_0,x_1 \neq 0$, $f$ admits at most $\deg(g)$ zeroes. Moreover, $f$ vanishes at $(0:1)$ if and only if $x_0 \mid f$, which is always the case when $w_1 \nmid d$. Similarly, $f$ vanishes at $(1:0)$ if and only if $x_1 \mid f$, which is always the case when $w_0 \nmid d$. This gives the expected bound.
\end{proof}

The denumerant $\den(d;w_0,w_1)$ appears in the bound in Proposition~\ref{prop:bound_m=1}. It is well known that this quantity can be made explicit using Euclidean division. The following result is well referenced in the case where $w_0$ and $w_1$ are coprime (e.g., see \cite[\textsection 4.4]{denumerant}). However, since we could not find a reference in the general case, we provide a short proof for the sake of completeness. 

\begin{lem}\label{lem:den}
    Let $(w_0,w_1) \in \NN_{\geq 1}^{2}$. Let $d \in \N$. Write the Euclidean division of $d$ by $\lcm(w_0,w_1)$ as 
        $d=\lambda \lcm (w_0,w_1) + \rho$ with $0\leq \rho< \lcm (w_0,w_1).$ Then
    \[\den(d;w_0,w_1)=\begin{cases}
        \lambda+\den(\rho;w_0,w_1) &\text{if } \gcd(w_0,w_1) \mid d,\\
        0 &\text{otherwise,}        
    \end{cases}\]
    with $\den(\rho;w_0,w_1) \in \set{0,1}$.
\end{lem}
\begin{proof}
      Set $\gamma=\gcd(w_0,w_1)$. When $\gamma=1$, the result follows from \cite[\textsection 4.4]{denumerant}. Now assume that $\gamma>1$. Then $\den(d;w_0,w_1)=0$ if $\gamma \nmid d$. Otherwise, write $w_0=\gamma w_0'$ and $w_1=\gamma w'_1$. The integer solutions of $w_0x+w_1 y= d$ are clearly the same as the ones of $w'_0x+w'_1 y= d/\gamma$, i.e., $\den(d;w_0,w_1)=\den(d/\gamma;w'_0,w'_1)$. To conclude, it suffices to notice that $\rho$ is equal to $\gamma$ times the remainder of the Euclidean division $d/\gamma$ by $w'_0w'_1$.

\end{proof}

Now we are ready to completely determine $e_q(d;w_0,w_1)$ for any degree $d\geq 1$. 

\begin{cor}\label{cor:m=1}
    Let $(w_0,w_1) \in \NN^2$ and $d\in \gen{w_0,w_1}_\NN$. Write the Euclidean division of $d$ by $\lcm(w_0,w_1)$ as
    \[
        d=\lambda \lcm (w_0,w_1) + \rho \text{ with } 0\leq \rho< \lcm (w_0,w_1).
    \]
     Then 
    \[e_q(d;w_0,w_1)=\min \set{q+1,
    \begin{cases}
    \lambda & \text{if } \lcm(w_0,w_1) \mid d,\\
    \lambda +1 & \text{if either } w_0 \mid d \text{ or } w_1 \mid d,\\
    \lambda +1+\den(\rho;w_0,w_1) & \text{if both } w_0 \nmid d \text{ and } w_1 \nmid d,
    \end{cases}}.
    \]
\end{cor}

\begin{proof}
The proof is two-fold. For each case, we use the formula for the denumerant from Lemma~\ref{lem:den} to check that the bound in Proposition~\ref{prop:bound_m=1} agrees with the one given here, and then we exhibit polynomials that reach this bound.

Set $\gamma=\gcd(w_0,w_1)$ and write $w_0=\gamma w'_0,$ $w_1=\gamma w'_1$. The idea is to build a polynomial as a product of as many ``linear factors'' of the type $x_1^{w_0'}-\alpha x_0^{w_1'}$, with $\alpha\in \fq$, and then, depending on the remainder $\rho$, see if we can construct a polynomial of degree $d$. We say these factors are ``linear'' since they have degree $\lcm(w_0,w_1)$, and by using Lemmas \ref{l:weightswithgcd} and \ref{l:delorme} we can map them to polynomials of degree $1$ in $\PP^1$. 

If both $w_0$ and $w_1$ divide $d$ (i.e., $\lcm(w_0,w_1) \mid d$) then $\rho=0$, which means $\den(\rho;w_0,w_1)=1$. We thus get 
\[\delta=\den(d;w_0,w_1)-1=\lambda+1-1=\lambda.\]
Take $J \subseteq \fq$ of cardinality $\abs{J}=\min\set{\lambda,q}$. The polynomial 
\[f=x_0^{(\lambda-\abs{J})w'_1}\prod_{\alpha \in J}(x_1^{w'_0}-\alpha x_0^{w'_1})\]
has degree $d$ and vanishes at $\abs{J}=\lambda$ $\fq$-points of $\PP(w_0,w_1)$ if $\lambda\leq q$, and on the whole set of $\fq$-points of $\PP(w_0,w_1)$ if $\lambda \geq q+1$.

\medskip

If either $w_0$ or $w_1$ divides $d$, but not both, then $\den(\rho;w_0,w_1)=1$ 
and
\[\delta+1=\den(d;w_0,w_1)=\lambda+1.\]

Take $J \subseteq \fq$ of cardinality $\abs{J}=\min\set{\lambda,q}$. If $w_0 \mid d$, then $w_0 \mid \rho$ and the polynomial
\[f=x_0^{\rho / w_0+(\lambda-\abs{J})w'_1}\prod_{\alpha \in J}(x_1^{w'_0}-\alpha x_0^{w'_1})\]
has degree $d$ and vanishes at $\abs{J}+1$ $\fq$-points of $\PP(w_0,w_1)$.
If $w_1 \mid d$, exchanging the role of the variables, one can easily check that the polynomial
\[f=x_1^{\rho/w_1+(\lambda-\abs{J})w'_0}\prod_{\alpha \in J}(x_0^{w'_1}-\alpha x_1^{w'_0})\]
has the same properties.

\medskip

If neither $w_0$ nor $w_1$ divides $d$, then 
\[\delta+2=\den(d;w_0,w_1)+1=\lambda+\den(\rho;w_0,w_1)+1.\]
\begin{itemize}
    \item If $\rho \in  \gen{w_0,w_1}_\NN$, then there exists a nonzero polynomial $g \in \fq[w_0,w_1]^w_\rho$. 
    Since neither $w_0$ nor $w_1$ divides $d$, and thus $\rho$, the polynomial $g$ is divisible by both $x_0$ and $x_1$. Take $J^\star \subseteq \fq^\star$ of cardinality $\min\set{\lambda,q-1}$. The polynomial 
\[f=gx_1^{(\lambda-\abs{J^\star})w'_0}\prod_{\alpha \in J^\star}(x_1^{w'_0}-\alpha x_0^{w'_1})\]
has degree $d$ and vanishes at $\abs{J^\star}+2=\min\set{\lambda +2,q+1}$ $\fq$-points $\PP(w_0,w_1)$.
     \item If $\rho \notin  \gen{w_0,w_1}_\NN$, then $\den(\rho+\lcm(w_0,w_1);w_0,w_1)\geq 1$ by Lemma \ref{lem:den}, which implies $\rho+\lcm(w_0,w_1) \in \gen{w_0,w_1}_\NN$. As before, there exists a nonzero polynomial $g \in \fq[w_0,w_1]^w_{\rho+\lcm(w_0,w_1)}$ that is divisible by both $x_0$ and $x_1$. Take $K^\star \subseteq \fq^\star$ of cardinality $\min\set{\lambda-1,q-1}$. The polynomial 
\[f=gx_1^{(\lambda-1-\abs{K^\star})w'_0}\prod_{\alpha \in K^\star}(x_1^{w'_0}-\alpha x_0^{w'_1})\]
has degree $d$ and vanishes at $\abs{K^\star}+2=\min\set{\lambda +1,q+1}$ $\fq$-points of $\PP(w_0,w_1)$.
\end{itemize}
\end{proof}

\begin{rem}\label{rem:w0_divides_d_formula}
    When $w_0$ divides $d$, we can gather the first two cases to get a single expression, noticing that 
    \[\Floorfrac{d-1}{\lcm(w_0,w_1)}+1=\begin{cases}
        \lambda &\text{if } w_1 \text{ also divides } d,\\
        \lambda+1 &\text{otherwise.}
        \end{cases}\]
\end{rem}

\subsection{Case $m>1$}

By Corollary~\ref{cor:m=1} and Remark~\ref{rem:w0_divides_d_formula}, taking a polynomial $f$ with maximum number of zeroes in $\PP(w_0,w_1)(\fq)$ and considering the associated variety in $\PP(w)$ gives
   \begin{equation}\label{eq:false_conj}
e_q(d;w_0,w_1,\dots,w_m)\geq \min\set{p_m,\left(\Floorfrac{d-1}{\lcm(w_0,w_1)}+1+\epsilon\right)q^{m-1}+p_{m-2}},    
   \end{equation}
   where $\epsilon=\begin{cases}
       \den(\rho;w_0,w_1) &\text{if } w_0 \nmid d \text{ and } w_1 \nmid d,\\
       0 &\text{otherwise.}
   \end{cases}$
   
   For $w_0=1$, we recover Theorem \ref{t:theorem_conjecture}, and we proved it to be an equality in this case. For $w_0 \geq 2$, we have found that equality does not always hold, i.e., it is no longer always true that there exist polynomials with maximum number of zeroes depending only on the variables $x_0$ and $x_1$, as in Equation \eqref{eq:min_pol}.

\begin{ex}\label{ex:w0>1}
    For $w=(2,3,5)$, we computed with {\scshape Magma} \cite{magma} all the polynomials with maximum number of zeroes over $\fq$ for different degrees $d$. For all the values of $q$ we tested, these are of the form 
    \[f=\begin{array}{lll}
        c x_0(x_2-\alpha x_0x_1) &\text{or } x_0^2x_1 & \text{for } d=7,\\
        c x_0x_1(x_2-\alpha x_0x_1) && \text{for } d=10,\\
        c x_0(x_2-\alpha x_0x_1)(x_2-\beta x_0x_1) &\text{or } c x_0^2x_1(x_2-\alpha x_0x_1) & \text{for } d=12,
    \end{array}        
    \]
    with $c \in \fq^\star$ and distinct $\alpha,\beta \in \fq$. The associated variety is always a union of lines, each containing $q+1$  $\fq$-points. However contrary to the case $w_0=1$, they do not always meet at a common point. Moreover, among the varieties with maximum number of $\fq$-points of degree $d=7$, we have $x_0^2x_1=0$, which consists of the union of a double line and a simple one. For $d=10$, these varieties are unions of three lines forming a triangle with the singular points $(1:0:0)$, $(0:1:0)$ and $(0:0:1)$ as vertices.
    For $d=12$ when $f=c x_0(x_2-\alpha x_0x_1)(x_2-\beta x_0x_1)$, the three lines $x_0=0$, $x_2=\alpha x_0x_1$ and $x_2=\beta x_0x_1$ all meet at $(0:1:0)$ and the two latter also meet at $(1:0:0)$.
    Hence we conjecture that
    \[e_q(d;2,3,5)=\begin{cases}
        2q+1 &\text{if }d=7,\\
        3q &\text{if }d=10 \text{ or } d=12.
    \end{cases}\]
    It is worth noting that, for $d\leq 11$ different from $7$ and $10$, equality holds in Equation \eqref{eq:false_conj} in all the test we ran. Since, for example, $3q$ is not of the form $aq^{m-1}+p_0=aq+1$, a simple modification of the coefficient of $q^{m-1}$ in the bound from Equation \eqref{eq:false_conj} would not be sufficient to obtain a general formula for the case $w_0>1$. This illustrates how intricate the case $w_0>1$ can be. 
\end{ex}

    Nevertheless, it is still possible to give some upper bounds on $e_q(d;w_0,w_1,\dots,w_m)$ under some hypotheses.

\begin{prop}\label{p:upper_bound}
Assume $\gcd(w_0,w_i,q-1)=1$, for all $1\leq i \leq m$. Then
$$
e_q(d;w_0,w_1,\dots,w_m)\leq \min\set{p_m,\left(\Floorfrac{d-1}{w_1}+1\right)q^{m-1}+p_{m-2}}.
$$
\end{prop}
\begin{proof}
By \cite[Prop. 3.4]{aubryperretWPRM}, under the given hypotheses we have
$$
\abs{V_{\PP(w)}(f)(\fq)}\leq \frac{1}{r_0}\sum_{i=0}^{r_0-1} \abs{V_{\PP(w^*)}(g_i)(\fq)},
$$
where $r_0=\gcd(w_0,q-1)$, $w^*=(1,w_1,\dots,w_m)$ and $g_i\in \fq[x_0,\dots,x_m]^{w^*}_d$. The stated bound follows from Theorem \ref{t:theorem_conjecture}.
\end{proof}
Unfortunately, this bounds seems far from being sharp. We give an example below.
\begin{ex}
Let $w=(2,3,5)$. For any even $q$, the conditions of Proposition \ref{p:upper_bound} are satisfied. For example, we get
$$
e_q(10;2,3,5)\leq 4q+1.
$$
However, this is not sharp since the monomial $x_0x_1x_2$ has $3q$ zeroes (recall Example \ref{ex:w0>1}).
\end{ex}

In our experiments, the lower bound from Equation (\ref{eq:false_conj}) is sharp in many cases for low degrees, while the upper bound from Proposition \ref{p:upper_bound} is not usually sharp.

\begin{ex}
    For $w=(2,3,5)$, \c{S}ahin gives a set of generators of $\PP(w)(\fq)$ \cite[Proposition~5.10]{mesutComputingVanishing}. The one with the smallest degree is $B_{0,1}$ (see Equation \eqref{eq:Bij}), with degree $w_0w_1(q-1)+w_0+w_1=6q-1$. For $5\leq d \leq 6q-2$ and $q\in \set{5,7,11}$, we have computed the minimum distance (or a range in which it lies) of the corresponding weighted projective Reed-Muller code using {\scshape Magma}, which is equal to $p_2-e_q(d;2,3,5)$ since the evaluation map is injective by the restriction on the degrees.
    We compare this value with the lower bound provided by Equation \eqref{eq:false_conj} in Table~\ref{tab:comparison}. One can see in Table \ref{tab:comparison} that the condition $\lcm(w)=30$ does not guarantee equality in Equation \eqref{eq:false_conj}.
\end{ex}

\begin{table}[ht]
    \centering
    \begin{minipage}[height=0.9\textheight]{.5\linewidth}
  \subfloat[$q=5$]{%
    \begin{tabular}{c|C{2.2cm}|C{2.2cm}}
    $d$&Lower bound Eq. \eqref{eq:false_conj}& $e_5(d;2,3,5)$ by {\scshape Magma} \\\hline
5,7,8,9 & \bf 11    & \bf 11 \\
6       & \bf 6     & \bf 6 \\
10,12   & 11        & 15 \\
11,13,14 & \bf 16   & \bf 16\\
15,16,18   & 16        & 19 \\
17,19   & \bf 21    & \bf 21 \\
20,22,24 & 21       & 23 \\
21      & 21        & 22 \\
23,26   & \bf 26    & \bf 26 \\
25,27,28 & 26       & 27 \\
\end{tabular}
}%
\vskip7.7em
  \subfloat[$q=7$]{%
    \begin{tabular}{c|C{2.2cm}|C{2.2cm}}
    $d$&Lower bound Eq. \eqref{eq:false_conj}& $e_7(d;2,3,5)$ by {\scshape Magma} \\\hline
5,7,8,9 &  \bf 15 &  \bf 15 \\
6       &  \bf 8 &  \bf 8 \\
10,12   & 15 & 21 \\
11,13,14& \bf 22 & \bf 22 \\
15,16,18& 22 &27 \\
17,19   & \bf 29 & \bf 29 \\
20,22,24& 29 & 33 \\
21      & 29 & 32 \\
23      & \bf 36 & \bf 36 \\
25,27,28& 36 & 39 \\
26      & 36 & 37 \\
29,31   & \bf 43 & \bf 43 \\
30      & 36 & 45 \\
32-34   & 43 & 45 \\
35      & 50 & 51 \\
36      & 43 & 47 \\
37-40   & 50 & 51 \\
\end{tabular}
}
    \end{minipage}%
    \begin{minipage}{.5\linewidth}
        \begin{subtable}{.5\linewidth}\centering
    \begin{tabular}{c|C{2.2cm}|C{2.2cm}}
    $d$&Lower bound Eq. \eqref{eq:false_conj}& $e_{11}(d;2,3,5)$ by {\scshape Magma} \\\hline
5,7,8,9 & \bf 23 & \bf 23 \\
6       & \bf 12 & \bf 12 \\
10,12   & 23 & 33 \\
11,13,14& \bf 34 & \bf 34 \\
15,16,18& 34 & 43 \\
17,19   & \bf 45 & \bf 45 \\
20,24   & 45 &(53-78*) \\
21      & 45 &(52-78*) \\
22      & 45 &(53-89*) \\
23      & 56 &(56-89*) \\
25      & 56 &(63-74) \\
26      & 56 &(61-79) \\
27      & 56 &(63-86) \\
28      & 56 &(63-89) \\
29      & 67 &(67-93) \\
30      & 56 &(73-97) \\
31      & 67 &(70-97) \\
32,33,34& 67 &(73-107) \\
35      & 78 &(83-112) \\
36      & 67 &(79-113) \\
37      & 78 &(77-113) \\
38,39   & 78 &(79-113) \\
40      & 78 &(83-117) \\
41,43,44,46& 89 &($\leq$ 89-118) \\
42      & 78 &(93-118) \\
45      & 89 &(103-118) \\
47,49   & 100 &($<$100-123) \\
48      & 89 &(91-123) \\
50,52,54& 100 &(113-123) \\
51      & 100 &(106-123) \\
53      & 111 &(113-123) \\
55,57      & 111 &123       \\
56      & 111 &(115-124) \\
58      & 111 &(123-126) \\
59      & 122 &(123-128) \\
60      & 111 &(123-128) \\
61,62,63,64& 122 &(124-128) \\
    \end{tabular}
    \caption{$q=11$}
        \end{subtable}
         \end{minipage}%
        \caption{$e_q(d;2,3,5)$ for $q=5, \: 7, \: 11$ and $5 \leq d \leq 6q-2$. The entries are bold when the lower bound equals the true value of $e_q(d;2,3,5)$. The starred upper bounds correspond to the value given in Proposition~\ref{p:upper_bound}, when the upper bound computed by {\scshape Magma} is larger.}
        \label{tab:comparison}
\end{table}

\section*{Acknowledgments}
The first author was supported  by the French National Research Agency through ANR \textit{Barracuda} (ANR-21-CE39-0009) and the French government \textit{Investissements d’Avenir} program ANR-11-LABX-0020-01. The second author was supported in part by the following grants: Grant PID2022-137283NB-C22 funded by MICIU/AEI/10.13039/501100011033 and by ERDF/EU, and FPU20/01311 funded by the Spanish Ministry of Universities. Part of this work was done during the visit of the second author to the University of Rennes. He thanks Jade Nardi and Gianira Alfarano for their hospitality.

\bibliographystyle{abbrv}


\appendix
\section{Isomorphisms of WPS}\label{s:appendix}

In the literature, one can find that $\PP(\gamma w)\cong \PP(w)$, and $\PP(w_0,w_1\gamma,\dots,w_m\gamma)\cong \PP(w)$ if $\gcd(w_0,\gamma)=1$. We provide now the explicit relation between the $\fq$-points of these projective spaces, as well as the corresponding coordinate rings.

\begin{lem}\label{l:weightswithgcd}
Let $w= (w_0,\dots,w_m)$. Let $\gamma$ be a positive integer and set $\gamma w= (w_0 \gamma,\dots,w_m\gamma)$. Then we have $\PP(\gamma w)(\fq)=\PP(w)(\fq)$ and  $\fq[x_0,\dots,x_m]_{\gamma d}^{\gamma w}=\fq[x_0,\dots,x_m]_d^w$ for every $d \geq 0$.
\end{lem}
\begin{proof}   
Let $\phi:\PP(\gamma w)\rightarrow \PP(w)$ be the inclusion map defined by $(Q_0:\cdots:Q_m)\mapsto (Q_0:\cdots:Q_m)$. Since $(\lambda^{w_0\gamma}Q_0:\cdots:\lambda^{w_m\gamma}Q_m)=(Q_0:\cdots:Q_m)$ in $\PP(w)$, for any $\lambda \in \overline{\fq}^*$, this map is well defined. This map is also injective, since $\phi(Q)=\phi(P)$ implies that there is $\lambda \in \overline{\fq}^*$ such that 
$$
(Q_0,\dots,Q_m)=(\lambda^{w_0}P_0,\dots,\lambda^{w_m}P_m).
$$
Let $\mu \in \overline{\fq}^*$ such that $\mu^\gamma=\lambda $. Then 
$$
(Q_0,\dots,Q_m)=(\mu^{w_0 \gamma}P_0,\dots,\mu^{w_m\gamma}P_m),
$$
that is, $Q=P$ in $\PP(\gamma w)$. It is now clear that $\phi$ induces a bijection between $\PP(\gamma w)(\fq)$ and $\PP(w)(\fq)$, since it sends $\fq$-points to $\fq$-points, it is injective, and the cardinality of both sets is $p_m$. Finally, a homogeneous polynomial of degree $d\gamma$ with weights $\gamma w$ can also be seen as a homogeneous polynomial of degree $d$ with weights $w$ (and vice versa).
\end{proof} 

Now we consider the isomorphism $\PP(w_0,w_1\gamma,\dots,w_m\gamma)\cong \PP(w)$, if $\gcd(w_0,\gamma)=1$, which is also called Delorme weight reduction \cite{delormeReduction}. We denote by $\varphi$ the map
$$
\begin{array}{lccc}
\varphi: &\PP(w_0,w_1\gamma,\dots,w_m\gamma) & \to & \PP(w) \\
& (Q_0:Q_1:\cdots:Q_m) & \mapsto & (Q_0^\gamma:Q_1:\cdots:Q_m).
\end{array}
$$
This map is well defined since $((\lambda^{w_0}Q_0)^\gamma:\lambda^{w_1\gamma}Q_1:\cdots:\lambda^{w_m\gamma}Q_m)=(Q_0^\gamma:Q_1:\cdots:Q_m)$ in $\PP(w)$, for any $\lambda \in \overline{\fq}^*$.

\begin{lem}\label{l:delorme}
Assume $\gcd(w_0,\gamma)=1$. We have 
\[\PP(w)(\fq)=\varphi(\PP(w_0,w_1\gamma,\dots,w_m\gamma)(\fq)).\]
Moreover, for any degree $d>0$, 
\[\varphi^* \fq[x_0,\dots,x_m]_d^w = \fq[x_0,\dots,x_m]_{\gamma d}^{(w_0,w_1\gamma,\dots,w_m\gamma)}.\]
\end{lem}
\begin{proof}
We start by showing that $\varphi$ is injective. Indeed, if $\varphi(Q)=\varphi(P)$, then there is $\lambda \in \overline{\fq}^*$ such that 
\begin{equation}\label{eq:delorme1}
(Q_0^\gamma,Q_1,\dots,Q_m)=(\lambda^{w_0}P_0^\gamma,\lambda^{w_1}P_1,\dots,\lambda^{w_m}P_m).
\end{equation}
First assume $Q_0\neq0 \neq P_0$. Since $\gcd(w_0,\gamma)=1$, there are integers $u,v$ such that $uw_0+v\gamma=1$. Let $\mu=(Q_0/P_0)^u \lambda^v \in \overline{\fq}^*$. Then we claim
$$
(Q_0,\dots,Q_m)=(\mu^{w_0}P_0,\mu^{w_1\gamma}P_1,\dots,\mu^{w_m\gamma}P_m).
$$
Indeed, we have to check that $\mu^\gamma=\lambda$ and $\mu^{w_0}=Q_0/P_0$. We have $\mu^\gamma=(Q_0/P_0)^{u\gamma } \lambda^{v\gamma}=\lambda^{uw_0+v\gamma}=\lambda$, since $\lambda^{w_0}=(Q_0/P_0)^\gamma$ (see Equation (\ref{eq:delorme1})). Similarly, $\mu^{w_0}=(Q_0/P_0)^{uw_0} \lambda^{vw_0}=(Q_0/P_0)^{uw_0+v\gamma}=Q_0/P_0$. If $Q_0=P_0=0$, then we take $\mu\in \overline{\fq}^*$ such that $\mu^\gamma=\lambda$. In both cases, we obtain that $Q=P$ in $\PP(w_0,w_1\gamma,\dots,w_m\gamma)$. 

With respect to $\fq$-points, it is clear that 
$$
\varphi(\PP(w_0,w_1\gamma,\dots,w_m\gamma)(\fq))\subseteq \PP(w)(\fq),
$$
and we have the equality by the injectivity of $\varphi$ and the fact that the number of $\fq$-points of both weighted projective spaces is $p_m$. 

Finally, regarding the last statement, let $f$ be a homogeneous polynomial of degree $d\gamma$ with weights $(w_0,w_1\gamma,\dots,w_m\gamma)$. Since $\gcd(w_0,\gamma)=1$, this implies that the exponents of $x_0$ appearing in the monomials of $f$ have to be multiples of $\gamma$. Let $g$ be the polynomial obtained by substituting every instance of $x_0^\gamma$ in the monomials of $f$ by $x_0$. Then $g$ is homogeneous of degree $d$ for the weights $w$, and 
$$
f(Q_0,\dots, Q_m)=g(Q_0^\gamma,Q_1,\dots, Q_m),
$$
for every $Q\in \PP(w)(\fq)$. This implies 
\[\varphi^* \fq[x_0,\dots,x_m]_d^w \supseteq \fq[x_0,\dots,x_m]_{\gamma d}^{(w_0,w_1\gamma,\dots,w_m\gamma)}\]
and the reverse inclusion is proved similarly by substituting instances of $x_0$ by $x_0^\gamma$. 
\end{proof}

In Lemma \ref{l:delorme}, the condition $\gcd(w_0,\gamma)=1$ is not a restriction, since otherwise we would have a nontrivial divisor of all the weights, and we may apply Lemma \ref{l:weightswithgcd}.

With respect to weighted projective lines, assuming $\gcd(w_0,w_1)=1$ (which we may always do in this case due to Lemma \ref{l:weightswithgcd}), Delorme weight reduction implies that $\PP(w_0,w_1)\cong \PP^1$. Consider the map
\begin{equation}\label{eq:isom_P(a,b)_P1}
\begin{array}{lccc}
\psi: &\PP(w_0,w_1) & \to & \PP(1,1)=\PP^1 \\
& (Q_0:Q_1) & \mapsto & (Q_0^{w_1}:Q_1^{w_0}).
\end{array}
\end{equation}
In terms of $\fq$-points, we have $\psi(\PP(w_0,w_1)(\fq))=\PP^1(\fq)$. 

\end{document}